\documentclass[12pt,reqno]{amsart}
\usepackage{amsthm,amsfonts,amssymb,euscript}

\def\normo#1{\left\|#1\right\|}
\def\normb#1{\big\|#1\big\|}

\def\aabs#1{\left|#1\right|}
\def\brk#1{\left(#1\right)}

\def\rev#1{\frac{1}{#1}}

\def\norm#1{\|#1\|}
\def\jb#1{\langle#1\rangle}
\def\wt#1{\widetilde{#1}}
\def\wh#1{\widehat{#1}}

\newcommand{\T}{{\mathbb T}}
\newcommand{\E}{{\mathbb E}}
\newcommand{\R}{{\mathbb R}}

\newcommand{\Z}{{\mathbb Z}}
\newcommand{\ft}{{\mathcal{F}}}

\newcommand{\les}{{\lesssim}}

\newcommand{\ra}{{\rightarrow}}
\newcommand{\Sch}{{\mathcal{S}}}
\newcommand{\supp}{{\mbox{supp}}}

\newcommand{\dr}{\omega}

\theoremstyle{plain}
  \newtheorem{theorem}[subsection]{Theorem}
  \newtheorem{proposition}[subsection]{Proposition}
  \newtheorem{lemma}[subsection]{Lemma}

\theoremstyle{remark}

\theoremstyle{definition}

\numberwithin{equation}{section}
\begin{document}

\title[Local regularity of the KP-I equation]{On the local regularity of the KP-I equation in anisotropic Sobolev space}
\author{Zihua Guo, Lizhong Peng, Baoxiang Wang}
\address{LMAM, School of Mathematical Sciences, Peking University, Beijing 100871, People's Republic of China.}
\email{zihuaguo, lzpeng, wbx@math.pku.edu.cn}


\begin{abstract}
We prove that the KP-I initial-value problem
\begin{eqnarray*}
\begin{cases}
\partial_tu+\partial_x^3u-\partial_x^{-1}\partial_y^2u+\partial_x(u^2/2)=0\text{ on }{\R}^2_{x,y}\times {\R}_t;\\
u(x,y,0)=\phi(x,y),
\end{cases}
\end{eqnarray*}
is locally well-posed in the space
\begin{eqnarray*}
H^{1,0}(\R^2)=\{\phi\in L^2(\R^2): \
\norm{\phi}_{H^{1,0}(\R^2)}\approx
\norm{\phi}_{L^2}+\norm{\partial_x\phi}_{L^2}<\infty\}.
\end{eqnarray*}
\end{abstract}

\maketitle


\section{Introduction}\label{section1}

In this paper we study the Cauchy problem for the KP-I equation
\begin{eqnarray}\label{eq:kpI}
\begin{cases}
\partial_tu+\partial_x^3u-\partial_x^{-1}\partial_y^2u+\partial_x(u^2/2)=0;\\
u(x,y,0)=\phi(x,y),
\end{cases}
\end{eqnarray}
where $u(x,y,t):\R^3 \rightarrow \R$ is the unknown function and
$\phi$ is the given initial data. The KP-I equation \eqref{eq:kpI}
and the KP-II equation (in which the sign of the term
$\partial_x^{-1}\partial_y^2u$ is $+$ instead of $-$) arise in
physical contexts as models for the propagation of dispersive long
waves with weak transverse effects.

Both of the KP-I and KP-II equations were widely studied. From the
point of view of well-posedness, the KP-II equation is better
understood than the KP-I equation. The main reason is that the KP-II
equation has a good geometric structure on the resonance while that
for the KP-I equation is much more complicated. It was proved by
Bourgain \cite{BourKP} that KP-II inital-value problem is globally
well-posed for suitable data in $L^2$ on both $\R^2$ and $\T^2$.
Takaoka and Tzvetkov \cite{TT} obtained local well-posedness for
data in anisotropic Sobolev space $H^{s_1,s_2}$, see also
\cite{Takaoka, Hadac} for the improvement to the full subcritical
cases. Their results were generalized by Hadac, Herr and Koch
\cite{HHK} to the sharp results in the critical space.

For the KP-I equation, it was proved by Molinet, Saut and Tzvetkov
\cite{MoSaTz1} that Picard iterative methods fail in standard
Sobolev space and $H^{s_1,s_2}$ for $s_1,s_2\in \R$, since it was
proved that the solution flow map fails to be $C^2$ smooth at the
origin in these spaces. However, in some other weighted spaces this
scheme still works, for example see \cite{CIKS}. By refined energy
methods, it is known that the KP-I equation is globally well-posed
in the "second" energy spaces, see \cite{Ke, MoSaTz2, MoSaTz3}.
Recently, Ionescu, Kenig and Tataru \cite{IKT} obtained global
well-posedness in the natural energy space $\E^1=\{\phi\in
L^2(\R^2),\partial_x \phi \in L^2(\R^2),
\partial_x^{-1}\partial_y \phi \in L^2(\R^2)\}$, by introducing some
new methods based on some earlier ideas in \cite{KoTa}, which can be
viewed as a combination of the Bourgain space $X^{s,b}$ methods and
energy estimates. Compared to the KP-II equation, well-posedness in
$L^2$ for the KP-I initial-value problem remains a challenging open
problem.

The purpose of this paper is to step forward in this direction. In
view of the results in the energy space $\E^1$, to prove the
well-posedness in $L^2$, the first step would be to remove the
condition $\partial_x^{-1}\partial_y \phi \in L^2$. Thus the natural
space for the initial data is now the anisotropic Sobolev space
$H^{s_1,s_2}$ which is defined by
\[
H^{s_1,s_2}=\{\phi \in L^2(\R^2):
\norm{\phi}_{H^{s_1,s_2}}=\norm{\wh{\phi}(\xi,\mu)(1+|\xi|^2)^{s_1/2}(1+|\mu|^2)^{s_2/2}}_{L^2_{\xi,\mu}}<\infty\}.
\]
We mainly consider the initial data in $H^{1,0}$, inspired by the
work of Molinet, Saut and Tzvetkov \cite{MoSaTz4} in which they
proved local well-posedness in $H^{s,0}$ for $s>3/2$.

Now we state our main results:

\begin{theorem}\label{thmmain}
(a) Assume $\phi\in H^{\infty,0}$. Then there exists
$T=T(\norm{u_0}_{H^{1,0}})>0$ such that there is a unique solution
$u=S^\infty_T(u_0)\in C([-T,T]:H^{\infty,0})$ of the Cauchy problem
\eqref{eq:kpI}. In addition, for any $\sigma\geq 1$
\begin{eqnarray}
\sup_{|t|\leq T} \norm{S^\infty_T(u_0)(t)}_{H^{\sigma,0}}\leq
C(T,\sigma,\norm{u_0}_{H^{\sigma,0}}).
\end{eqnarray}

(b) Moreover, the mapping $S_T^{\infty,0}:H^{\infty,0}\rightarrow
C([-T,T]:H^{\infty,0})$ extends uniquely to a continuous mapping
\[S_T^s:H^{1,0} \rightarrow C([-T,T]:H^{1,0}).\]
\end{theorem}

We sketch the proof of the main theorem. Since the direct iterative
methods can not work for the data in $H^{s_1,s_2}$, one then need to
use some other less perturbative methods. We will use the methods of
Ionescu, Kenig and Tataru \cite{IKT} in this paper. It is known that
one can gain one-half derivative by direct using of Bourgain space
$X^{s,b}$ (see \cite{MoSaTz1}, also see Section 3 below) associated
to the initial data $H^{s,0}$. Thus in order to gain one derivative,
we use the short-time $X^{s,b}$-structure. The time scale used in
this paper seems to be optimal in the sense that it just suffices to
control the high-low interaction, see \cite{GuodgBO} for a
discussion of the optimal time scale. The next step is to prove an
energy estimates which are used to control the rest. Finally using
the arguments in \cite{IKT} and Bona-Smith methods \cite{BonaSmith},
we are able to prove the continuity of the solution mapping. The
arguments here are less complicated than those in \cite{IKT}, since
we don't have the difficulty from the low frequency part and can
just treat the equation as if it is of spatial dimension one.

It would be also very interesting if one can remove the condition
$\partial_x \phi \in L^2(\R^2)$ and keep the condition
$\partial_x^{-1}\partial_y \phi \in L^2(\R^2)$. Actually, the
arguments in \cite{IKT} have proved better results than the energy
space $\E^1$. For example, the whole arguments actually work for the
initial data belonging to a larger space $\E^\sigma=\{\phi\in
L^2(\R^2),|\partial_x|^{\sigma} \phi \in L^2(\R^2),
\partial_x^{-1}\partial_y \phi \in L^2(\R^2)\}$ for some $\sigma<1$.
But they may not work for $\E^0$.

The rest of the paper is organized as following. In Section 2 we
present some notations and Banach function spaces. We present a
symmetric estimtes in Section 3 and use them to prove some dyadic
bilinear estimates in Section 4. The proof of Theorem \ref{thmmain}
is given in Section 5 via some energy estimates which are proved in
Section 6.

\section{Notation and Definitions}

For $x, y\in \R^+$, $x\les y$ means that there exists $C>0$ such
that $x\leq Cy$. By $x\sim y$ we mean $x\les y$ and $y\les x$. For
$f\in \Sch'$ we denote by $\widehat{f}$ or $\ft (f)$ the Fourier
transform of $f$ for both spatial and time variables,
\begin{eqnarray*}
\widehat{f}(\xi,\mu,\tau)=\int_{\R^3}e^{-ix \xi}e^{-iy\mu}e^{-it
\tau}f(x,y,t)dxdydt.
\end{eqnarray*}
Moreover, we use $\ft_{x,y}$ and $\ft_t$ to denote the Fourier
transform with respect to space and time variable respectively. Let
$\Z_+=\Z\cap[0, \infty)$. Let $I_{\leq 0}=\{\xi: |\xi|<3/2\}$,
$\wt{I}_{\leq 0}=\{\xi:|\xi|\leq 2\}$. For $k\in \Z$ let
\[I_k=\{\xi:|\xi|\in [(3/4)\cdot 2^k,
(3/2)\cdot 2^k)\ \}\mbox{ and }\widetilde{I}_k=\{\xi: |\xi|\in
[2^{k-1}, 2^{k+1}]\ \}.\]

Let $\eta_0: \R\rightarrow [0, 1]$ denote an even smooth function
supported in $[-8/5, 8/5]$ and equal to $1$ in $[-5/4, 5/4]$. For
$k\in \Z$ let $\chi_k(\xi)=\eta_0(\xi/2^k)-\eta_0(\xi/2^{k-1})$ and
\[\chi_{[k_1,k_2]}=\sum_{k=k_1}^{k_2}\chi_k \mbox{ for any } k_1\leq k_2\in \Z.\]
For simplicity, let $\eta_k=\chi_k$ if $k\geq 1$ and $\eta_k\equiv
0$ if $k\leq -1$. Also, for $k_1\leq k_2\in \Z$ let
\[\eta_{[k_1,k_2]}=\sum_{k=k_1}^{k_2}\eta_k \mbox{ and }\eta_{\leq k_2}=\sum_{k=-\infty}^{k_2}\eta_{k}.\]
Roughly speaking, $\{\chi_k\}_{k\in \mathbb{Z}}$ is the homogeneous
decomposition function sequence and $\{\eta_k\}_{k\in \mathbb{Z}_+}$
is the non-homogeneous decomposition function sequence to the
frequency space. For $k\in \Z$ let $P_k$ denote the operators on
$L^2(\R^2)$ defined by
$\widehat{P_ku}(\xi)=1_{I_k}(\xi)\widehat{u}(\xi,\mu)$. By a slight
abuse of notation we also define the operators $P_k$ on
$L^2(\R^2\times \R)$ by formulas $\ft(P_ku)(\xi,\mu,
\tau)=1_{I_k}(\xi)\ft(u)(\xi,\mu, \tau)$. For $l\in \Z$ let
\[
P_{\leq l}=\sum_{k\leq l}P_k, \quad P_{\geq l}=\sum_{k\geq l}P_k.
\]

For $x\in \R$, let $[x]$ be the largest integer that is less or
equal to $x$ and denote $\jb{x}=(1+|x|^2)^{1/2}$. Let $a_1, a_2,
a_3\in \R$. It will be convenient to define the quantities
$a_{max}\geq a_{med}\geq a_{min}$ to be the maximum, median, and
minimum of $a_1,a_2,a_3$ respectively. Usually we use $k_1,k_2,k_3$
and $j_1,j_2,j_3$ to denote integers, $N_i=2^{k_i}$ and
$L_i=2^{j_i}$ for $i=1,2,3$ to denote dyadic numbers.

For $(\xi,\mu)\in \R \setminus \{0\}\times \R$ let
\begin{equation}\label{eq:dr}
\dr(\xi,\mu)=\xi^3+\mu^2/\xi.
\end{equation}
which is the dispersion relation associated to KP-I equation
\eqref{eq:kpI}. For $\phi \in L^2(\R^2)$ let $W(t)\phi\in C(\R:
L^2)$ denote the solution of the free KP-I evolution given by
\begin{eqnarray*}
\ft_{x,y}[W(t)\phi](\xi,\mu,t)=e^{it\omega(\xi,\mu)}\widehat{\phi}(\xi,\mu).
\end{eqnarray*}
For $k,j \in \Z_+$ let \[D_{k,j}=\{(\xi,\mu, \tau)\in \R^3: \xi \in
\widetilde{I}_k, \tau-\omega(\xi,\mu)\in \widetilde{I}_j\},\quad
D_{k,\leq j}=\cup_{l\leq j}D_{k,l}.\]Similarly we define $D_{\leq
k,j}$ and $D_{\leq k,\leq j}$. For $k\in \Z_+$ we define the dyadic
$X^{s,b}$-type normed spaces $X_k(\R^3)$:
\begin{eqnarray*}
&X_k=\left\{
\begin{array}{l}
f\in L^2(\R^3): f(\xi,\mu,\tau) \mbox{ is supported in } \widetilde
{I}_k\times\R^2 \mbox{ ($\wt{I}_{\leq 0}\times \R^2$ if $k=0$)}\\
\mbox{ and }\norm{f}_{X_k}:=\sum_{j=0}^\infty
2^{j/2}\norm{\eta_j(\tau-w(\xi,\mu))\cdot
f(\xi,\mu,\tau)}_{L^2_{\xi,\mu,\tau}}<\infty
\end{array}
\right\}.&
\end{eqnarray*}
These $l^1$-type $X^{s,b}$ structures were first introduced in
\cite{Tata}. Our resolution space is a little different from those
in \cite{IKT}. We do not perform the homogeneous decomposition on
the low frequency. In this way we can avoid some logarithmic
divergence.

The definition shows easily that if $k\in \Z_+$ and $f_k\in X_k$
then
\begin{eqnarray}\label{eq:pXk1}
\normo{\int_{\R}|f_k(\xi,\mu,\tau')|d\tau'}_{L_\xi^2}\les
\norm{f_k}_{X_k}.
\end{eqnarray}
Moreover, it is easy to see (see \cite{IKT}, see \cite{GuodgBO} for
a proof) that if $k\in \Z_+$, $l\in \Z_+$, and $f_k\in X_k$ then
\begin{eqnarray}\label{eq:pXk2}
&&\sum_{j=l+1}^\infty
2^{j/2}\normo{\eta_j(\tau-\omega(\xi,\mu))\cdot
\int_{\R}|f_k(\xi,\mu,\tau')|\cdot
2^{-l}(1+2^{-l}|\tau-\tau'|)^{-4}d\tau'}_{L^2}\nonumber\\
&&+2^{l/2}\normo{\eta_{\leq l}(\tau-\omega(\xi,\mu)) \int_{\R}
|f_k(\xi,\mu,\tau')|
2^{-l}(1+2^{-l}|\tau-\tau'|)^{-4}d\tau'}_{L^2}\nonumber\\
&&\les \norm{f_k}_{X_k}.
\end{eqnarray}
In particular, if $k,\ l\in \Z_+$, $t_0\in \R$, $f_k \in X_k$ and
$\gamma\in \Sch(\R)$, then
\begin{eqnarray}\label{eq:pXk3}
\norm{\ft[\gamma(2^l(t-t_0))\cdot \ft^{-1}(f_k)]}_{X_k}\les
\norm{f_k}_{X_k}.
\end{eqnarray}

As in \cite{IKT} at frequency $2^k$ we will use the $X^{s, b}$
structure given by the $X_k$ norm, uniformly on the $2^{-k}$ time
scale. For $k\in \Z_+$ we define the normed spaces
\begin{eqnarray*}
&& F_k=\left\{
\begin{array}{l}
f\in L^2(\R^3): \widehat{f}(\xi,\mu,\tau) \mbox{ is supported in }
\widetilde{I}_k\times\R^2 \mbox{ ($\wt{I}_{\leq 0}\times \R^2$ if $k=0$)} \\
\mbox{ and }\norm{f}_{F_k}=\sup\limits_{t_k\in \R}\norm{\ft[f\cdot
\eta_0(2^{k}(t-t_k))]}_{X_k}<\infty
\end{array}
\right\},
\\
&&N_k=\left\{
\begin{array}{l}
f\in L^2(\R^3): \supp \widehat{f}(\xi,\mu,\tau) \subset
\widetilde{I}_k\times\R^2 \mbox{ ($\wt{I}_{\leq 0}\times \R^2$ if $k=0$)} \mbox{ and } \\
\norm{f}_{N_k}=\sup\limits_{t_k\in
\R}\norm{(\tau-\dr(\xi,\mu)+i2^{k})^{-1}\ft[f\cdot
\eta_0(2^{k}(t-t_k))]}_{X_k}<\infty
\end{array}
\right\}.
\end{eqnarray*}
We see from the definitions that we still use $X^{s,b}$ structure on
the whole interval for the low frequency. We define then local
versions of the spaces in standard ways. For $T\in (0,1]$ we define
the normed spaces
\begin{eqnarray*}
F_k(T)&=&\{f\in C([-T,T]:L^2): \norm{f}_{F_k(T)}=\inf_{\wt{f}=f
\mbox{ in } \R\times [-T,T]}\norm{\wt f}_{F_k}\};\\
N_k(T)&=&\{f\in C([-T,T]:L^2): \norm{f}_{N_k(T)}=\inf_{\wt{f}=f
\mbox{ in } \R\times [-T,T]}\norm{\wt f}_{N_k}\}.
\end{eqnarray*}
We assemble these dyadic spaces in a Littlewood-Paley manner. For
$s\geq 0$ and $T\in (0,1]$, we define the normed spaces
\begin{eqnarray*}
&&F^{s}(T)=\left\{ u:\
\norm{u}_{F^{s}(T)}^2=\sum_{k=1}^{\infty}2^{2sk}\norm{P_k(u)}_{F_k(T)}^2+\norm{P_{\leq
0}(u)}_{F_0(T)}^2<\infty \right\},
\\
&&N^{s}(T)=\left\{ u:\
\norm{u}_{N^{s}(T)}^2=\sum_{k=1}^{\infty}2^{2sk}\norm{P_k(u)}_{N_k(T)}^2+\norm{P_{\leq
0}(u)}_{N_0(T)}^2<\infty \right\}.
\end{eqnarray*}
We define the dyadic energy space. For $s\geq 0$ and $u\in
C([-T,T]:H^{\infty,0})$ we define
\begin{eqnarray*}
\norm{u}_{E^{s}(T)}^2=\norm{P_{\leq 0}(u(0))}_{L^2}^2+\sum_{k\geq
1}\sup_{t_k\in [-T,T]}2^{2sk}\norm{P_k(u(t_k))}_{L^2}^2.
\end{eqnarray*}
As in \cite{IKT}, for any $k\in \Z_+$ we define the set $S_k$ of
$k-acceptable$ time multiplication factors
\begin{eqnarray}
S_k=\{m_k:\R\rightarrow \R: \norm{m_k}_{S_k}=\sum_{j=0}^{10}
2^{-k}\norm{\partial^jm_k}_{L^\infty}< \infty\}.
\end{eqnarray}
For instance, $\eta(2^{k}t) \in S_k$ for any $\eta$ satisfies
$\norm{\partial_x^j \eta}_{L^\infty}\leq C$ for $j=0,1,2,\ldots,
10$. Direct estimates using the definitions and \eqref{eq:pXk2} show
that for any $s\geq 0$ and $T\in (0,1]$
\begin{eqnarray}\label{eq:Sk}
\left \{
\begin{array}{l}
\norm{\sum_{k\in \Z_+} m_k(t)\cdot P_k(u)}_{F^{s}(T)}\les (\sup_{k\in \Z_+}\norm{m_k}_{S_k})\cdot \norm{u}_{F^{s}(T)};\\
\norm{\sum_{k\in \Z_+} m_k(t)\cdot P_k(u)}_{N^{s}(T)}\les
(\sup_{k\in \Z_+}\norm{m_k}_{S_k})\cdot \norm{u}_{N^{s}(T)};\\
\norm{\sum_{k\in \Z_+} m_k(t)\cdot P_k(u)}_{E^{s}(T)}\les
(\sup_{k\in \Z_+}\norm{m_k}_{S_k})\cdot \norm{u}_{E^{s}(T)}.
\end{array}
\right.
\end{eqnarray}
We refer the readers to \cite{GuodgBO} for some detailed proof.

\section{$L^2$ bilinear estimates}\label{L2bi}

In this section we prove some symmetric estimates by following some
ideas in \cite{IKT} and \cite{MoSaTz1}. These estimates will be used
to prove bilinear estimates in the next section.

\begin{lemma}\label{lem:3zest}
(a) Assume $k_1,k_2,k_3\in \Z$, $j_1,j_2,j_3\in\Z_+$, and
$f_i:\R^3\to\R_+$ are $L^2$ functions supported in $D_{k_i,\leq
j_i}$, $i=1,2,3$. Then
\begin{eqnarray}\label{eq:3zesta}
\int_{\R^3}(f_1\ast f_2)\cdot f_3\lesssim
2^{\frac{j_1+j_2+j_3}{2}}\min(2^{-\frac{k_1+k_2+k_3}{2}},2^{-\frac{j_{max}}{2}})
\|f_1\|_{L^2}\|f_2\|_{L^2}\|f_3\|_{L^2}.
\end{eqnarray}

(b) Assume $k_1,k_2,k_3\in \Z$ with $k_2\geq 20$, $|k_2-k_3|\leq 5$
and $k_1\leq k_2-10$, $j_1,j_2,j_3\in\Z_+$, and $f_i:\R^3\to\R_+$
are $L^2$ functions supported in $D_{k_i,\leq j_i}$, $i=1,2,3$. Then
if $j_{max}\leq k_1+k_2+k_3-20$,
\begin{eqnarray}\label{eq:3zestb1}
\int_{\R^3}(f_1\ast f_2)\cdot f_3\lesssim   2^{(j_1+j_2)/2}
2^{-k_3/2}2^{k_1/2}\|f_1\|_{L^2}\|f_2\|_{L^2}\|f_3\|_{L^2};
\end{eqnarray}
or else if $j_{max}\geq k_1+k_2+k_3-20$,
\begin{eqnarray}\label{eq:3zestb2}
\int_{\R^3}(f_1\ast f_2)\cdot f_3\lesssim
2^{(j_1+j_2)/2}2^{j_{max}/4}
2^{-k_3}2^{k_1/4}\|f_1\|_{L^2}\|f_2\|_{L^2}\|f_3\|_{L^2}.
\end{eqnarray}
\end{lemma}

\begin{proof}
Part (a) follows immediately from Lemma 5.1 (a) and Lemma 5.2 in
\cite{IKT}. We only need to prove part (b). Define
$f_i^\#(\xi,\mu,\theta)=f_i(\xi,\mu,\theta+\omega(\xi,\mu))$,
$i=1,2,3$, then we see $\|f_i^\#\|_{L^2}=\|f_i\|_{L^2}$ and the
functions $f_i^\#$ are supported in $\{\xi,\mu,\theta):\,\xi\in
\widetilde{I}_{k_i},\,\mu\in\R,\,|\theta|\leq 2^{j_i}\}$. We rewrite
the left-hand side of \eqref{eq:3zestb1}
\begin{eqnarray*}
&&\int_{\R^6}f_1^\#(\xi_1,\mu_1,\theta_1)\cdot f_2^\#(\xi_2,\mu_2,\theta_2)\\
&&\times
f_3^\#(\xi_1+\xi_2,\mu_1+\mu_2,\theta_1+\theta_2+\Omega((\xi_1,\mu_1),(\xi_2,\mu_2)))\,d\xi_1d\xi_2d\mu_1d\mu_2d\theta_1d\theta_2,
\end{eqnarray*}
where
\begin{eqnarray*}
\Omega((\xi_1,\mu_1),(\xi_2,\mu_2))&=&-\omega(\xi_1+\xi_2,\mu_1+\mu_2)+\omega(\xi_1,\mu_1)+\omega(\xi_2,\mu_2)\\
&=&\frac{-\xi_1\xi_2}{\xi_1+\xi_2}\Big[(\sqrt{3}\xi_1+\sqrt{3}\xi_2)^2-\Big(\frac{\mu_1}{\xi_1}-\frac{\mu_2}{\xi_2}\Big)^2\Big]
\end{eqnarray*}
is the resonance function for KP-I equation. For the KP-II equation,
the resonance function is
$\frac{-\xi_1\xi_2}{\xi_1+\xi_2}\Big[(\sqrt{3}\xi_1+\sqrt{3}\xi_2)^2+\Big(\frac{\mu_1}{\xi_1}-\frac{\mu_2}{\xi_2}\Big)^2\Big]$
instead, thus we see why KP-II is easier to handle.

We assume first that $j_{max}\leq k_1+k_2+k_3-20$. For this case, we
reproduce part of the proof of Lemma 5.1 (a) in \cite{IKT}. We will
prove that if $g_i:\R^2\to\R_+$ are $L^2$ functions supported in
$\widetilde{I}_{k_i}\times\R$, $i=1,2$, and $g:\R^3\to\R_+$ is an
$L^2$ function supported in $\widetilde{I}_{k_3}\times\R\times
[-2^{j_{max}+5},2^{j_{max}+5}]$, then
\begin{eqnarray}\label{eq:3zestbprf1}
&&\int_{\R^4}g_1(\xi_1,\mu_1) g_2(\xi_2,\mu_2)g(\xi_1+\xi_2,\mu_1+\mu_2,\Omega((\xi_1,\mu_1),(\xi_2,\mu_2)))\,d\xi_1d\xi_2d\mu_1d\mu_2\nonumber\\
&&\lesssim
2^{k_1/2}2^{-k_3/2}\cdot\|g_1\|_{L^2}\|g_2\|_{L^2}\|g\|_{L^2}.
\end{eqnarray}
This suffices for \eqref{eq:3zestb1}, combined with Cauchy-Schwartz
inequality.

To prove \eqref{eq:3zestbprf1}, we may assume that the integral in
the left-hand side of \eqref{eq:3zestbprf1} is taken over the set
(there are four identical integrals of this type)
\begin{equation*}
\mathcal{R}_{++}=\{(\xi_1,\mu_1,\xi_2,\mu_2):\,\xi_1+\xi_2\geq
0\text{ and }\mu_1/ \xi_1-\mu_2/ \xi_2\geq 0\}.
\end{equation*}
From the assumption $j\leq k_1+k_2+k-20$, we may assume that the
integral in the left-hand side of \eqref{eq:3zestbprf1} is taken
over the set
\begin{equation*}
\widetilde{\mathcal{R}}_{++}=\{(\xi_1,\mu_1,\xi_2,\mu_2)\in\mathcal{R}_{++}:|\sqrt{3}(\xi_1+\xi_2)|-|\mu_1/
\xi_1-\mu_2/ \xi_2|\leq 2^{-10}|\xi_1+\xi_2|\}.
\end{equation*}
To summarize, it suffices to prove that
\begin{eqnarray}\label{eq:3zestbprf2}
&&\int_{\widetilde{\mathcal{R}}_{++}}g_1(\xi_1,\mu_1)
g_2(\xi_2,\mu_2)
g(\xi_1+\xi_2,\mu_1+\mu_2,\Omega((\xi_1,\mu_1),(\xi_2,\mu_2)))\,d\xi_1d\xi_2d\mu_1d\mu_2\nonumber\\
&&\lesssim 2^{k_1/2}2^{-k_3/2}
\cdot\|g_1\|_{L^2}\|g_2\|_{L^2}\|g\|_{L^2}.
\end{eqnarray}
By the changes of variables
\begin{equation*}
\mu_1=\sqrt{3}\xi_1^2+\beta_1\xi_1\text{ and
}\mu_2=-\sqrt{3}\xi_2^2+\beta_2\xi_2,
\end{equation*}
with $d\mu_1d\mu_2=\xi_1\xi_2\,d\beta_1d\beta_2$. The left-hand side
of \eqref{eq:3zestbprf2} is bounded by
\begin{eqnarray*}
&&C2^{k_1+k_2}\int_{S}
g(\xi_1+\xi_2,\sqrt{3}\xi_1^2-\sqrt{3}\xi_2^2+\beta_1\xi_1+\beta_2\xi_2,\widetilde{\Omega}((\xi_1,\beta_1),(\xi_2,\beta_2)))\nonumber\\
&&\quad \quad g_1(\xi_1,\sqrt{3}\xi_1^2+\beta_1\xi_1)\cdot
g_2(\xi_2,-\sqrt{3}\xi_2^2+\beta_2\xi_2)d\xi_1d\xi_2d\beta_1d\beta_2,
\end{eqnarray*}
where
\begin{eqnarray*}
S=\{(\xi_1,\beta_1,\xi_2,\beta_2):\xi_1+\xi_2\geq 0\text{ and
}|\beta_1-\beta_2|\leq 2^{-10}(\xi_1+\xi_2)\},
\end{eqnarray*}
and
\begin{eqnarray*}
\widetilde{\Omega}((\xi_1,\beta_1),(\xi_2,\beta_2))=\xi_1\xi_2
(\beta_1-\beta_2)\Big(2\sqrt{3}+\frac{\beta_1-\beta_2}{\xi_1+\xi_2}\Big).
\end{eqnarray*}
Let
$h_1(\xi_1,\beta_1)=2^{k_1/2}g_1(\xi_1,\sqrt{3}\xi_1^2+\beta_1\xi_1)$,
$h_2(\xi_2,\beta_2)=2^{k_2/2}g_2(\xi_2,-\sqrt{3}\xi_2^2+\beta_2\xi_2)$,
then $\|h_i\|_{L^2}\approx\|g_i\|_{L^2}$. Thus, for
\eqref{eq:3zestbprf2} it suffices to prove that
\begin{eqnarray}\label{eq:3zestbprf6}
&&2^{(k_1+k_2)/2}\int_{S}g(\xi_1+\xi_2,\sqrt{3}\xi_1^2-\sqrt{3}\xi_2^2+\beta_1\xi_1+\beta_2\xi_2,\widetilde{\Omega}((\xi_1,\beta_1),(\xi_2,\beta_2)))\nonumber\\
&&\cdot h_1(\xi_1,\beta_1)\cdot
h_2(\xi_2,\beta_2)d\xi_1d\xi_2d\beta_1d\beta_2\les
2^{k_1/2}2^{-k_3/2}\cdot\|h_1\|_{L^2}\|h_2\|_{L^2}\|g\|_{L^2}.
\end{eqnarray}
By the change of variables $\beta_1=\beta_2+\beta$, we get that the
left-hand side of \eqref{eq:3zestbprf6} is bounded by
\begin{eqnarray}\label{eq:3zestbprf7}
&&\int_{\widetilde{S}}g(\xi_1+\xi_2,\sqrt{3}\xi_1^2-\sqrt{3}\xi_2^2+\beta
\xi_1+\beta_2(\xi_1+\xi_2),\xi_1\xi_2\beta\cdot
(2\sqrt{3}+\beta/(\xi_1+\xi_2)))\nonumber\\
&&\times 2^{(k_1+k_2)/2}h_1(\xi_1,\beta+\beta_2)\cdot
h_2(\xi_2,\beta_2) d\xi_1d\xi_2d\beta d\beta_2,
\end{eqnarray}
where
$\widetilde{S}=\{(\xi_1,\xi_2,\beta,\beta_2)\in\R^4:\xi_1+\xi_2\geq
0\text{ and }|\beta|\leq 2^{-10}(\xi_1+\xi_2)\}$. Note that in the
area $\widetilde{S}$ we have
\[|\partial_{\beta_2}[\sqrt{3}\xi_1^2-\sqrt{3}\xi_2^2+\beta
\xi_1+\beta_2(\xi_1+\xi_2)]|\sim 2^{k_3}\] and
\[|\partial_{\beta}[\xi_1\xi_2\beta\cdot
(2\sqrt{3}+\beta/(\xi_1+\xi_2))]|\sim 2^{k_3+k_1},\] thus using
Cauchy-Schwartz inequality for $\beta_2, \beta, \xi_2, \xi_1$
respectively, then we get that \eqref{eq:3zestbprf7} is bounded by
\[
2^{k_1/2}2^{-k_3/2}\cdot\|h_1\|_{L^2}\|h_2\|_{L^2}\|g\|_{L^2}\]
which gives the bound \eqref{eq:3zestbprf6}, as desired.

We assume now $j_{max}\geq k_1+k_2+k_3-20$. As the previous case, we
will prove that if $g_i:\R^2\to\R_+$ are $L^2$ functions supported
in $\widetilde{I}_{k_i}\times\R$, $i=1,2$, and $g:\R^3\to\R_+$ is an
$L^2$ function supported in $\widetilde{I}_{k_3}\times\R\times
[-2^{j_{max+5}},2^{j_{max}+5}]$, then
\begin{eqnarray}\label{eq:3zestcprf1}
&&\int_{\R^4}g_1(\xi_1,\mu_1) g_2(\xi_2,\mu_2)g(\xi_1+\xi_2,\mu_1+\mu_2,\Omega((\xi_1,\mu_1),(\xi_2,\mu_2)))\,d\xi_1d\xi_2d\mu_1d\mu_2\nonumber\\
&&\lesssim
2^{-k_3}2^{k_1/4}2^{j_{max}/4}\cdot\|g_1\|_{L^2}\|g_2\|_{L^2}\|g\|_{L^2}.
\end{eqnarray}
This suffices for \eqref{eq:3zestb2}, combined with Cauchy-Schwartz
inequality.

By a change of variable $\xi_2'=\xi_2+\xi_1,\ \mu_2'=\mu_2+\mu_1$ we
get the left-hand side of \eqref{eq:3zestcprf1} is dominated by
\begin{eqnarray}\label{eq:3zestcprf2}
&&\int_{\R^4}g_1(\xi_1,\mu_1)
g_2(\xi_2-\xi_1,\mu_2-\mu_1)\nonumber\\
&&\times \
g(\xi_2,\mu_2,\Omega((\xi_1,\mu_1),(\xi_2-\xi_1,\mu_2-\mu_1)))\,d\xi_1d\xi_2d\mu_1d\mu_2.
\end{eqnarray}
Using Cauchy-Schwartz inequality we get that \eqref{eq:3zestcprf2}
is dominated by
\begin{eqnarray}\label{eq:3zestcprf3}
&&\int_{\R^2}\brk{\int_{\R^2}|g_1(\xi_1,\mu_1)
g_2(\xi_2-\xi_1,\mu_2-\mu_1)|^2d\xi_1 d\mu_1}^{1/2}\nonumber\\
&&\times \brk{\int_{\R^2}
|g(\xi_2,\mu_2,\Omega((\xi_1,\mu_1),(\xi_2-\xi_1,\mu_2-\mu_1)))|^{2}d\xi_1d\mu_1}^{1/2}d\xi_2d\mu_2.
\end{eqnarray}
Make a change of variable
\[\alpha=\Omega((\xi_1,\mu_1),(\xi_2-\xi_1,\mu_2-\mu_1)), \ \beta=3\xi_2\xi_1(\xi_2-\xi_1),\]
then we see $|\alpha|\les 2^{j_{max}},\ |\beta|\les 2^{j_{max}}, \
|\beta|\les 2^{k_1+2k_2}$ and
\[d\xi_1 d\mu_1=\frac{|\beta|^{1/2}d\alpha d\beta}{|\xi_2|^{3/2}|\frac{3\xi_2^3}{4}-\beta|^{1/2}|\beta+\alpha|^{1/2}}.\]
Thus we get
\begin{eqnarray*}
&&\brk{\int_{\R^2}
|g(\xi_2,\mu_2,\Omega((\xi_1,\mu_1),(\xi_2-\xi_1,\mu_2-\mu_1)))|^{2}d\xi_1d\mu_1}^{1/2}\\
&&\les \brk{\int_{\R^2}
|g(\xi_2,\mu_2,\alpha)|^{2}\frac{|\beta|^{1/2}}{|\xi_2|^{3/2}|\frac{3\xi_2^3}{4}-\beta|^{1/2}|\beta+\alpha|^{1/2}}d\alpha
d\beta}^{1/2}\\
&&\les
2^{-k_3}2^{k_1/4}2^{j_{max}/4}\norm{g(\xi_2,\mu_2,\cdot)}_{L^2}.
\end{eqnarray*}
Therefore, we get
\[\eqref{eq:3zestcprf3}\les 2^{-k_3}2^{k_1/4}2^{j_{max}/4}\|g_1\|_{L^2}\|g_2\|_{L^2}\|g\|_{L^2}\]
which suffices to give the bound \eqref{eq:3zestb2}. By
interpolating \eqref{eq:3zesta} and \eqref{eq:3zestb2} then we get
that: under the same condition as for \eqref{eq:3zestb2} we have
\begin{eqnarray}\label{eq:3zestbextra}
\int_{\R^3}(f_1\ast f_2)\cdot f_3\lesssim 2^{(j_1+j_2)/2}2^{j_{3}/6}
2^{-2k_3/3}2^{k_1/6}\|f_1\|_{L^2}\|f_2\|_{L^2}\|f_3\|_{L^2}.
\end{eqnarray}
We will use it in the sequel.
\end{proof}

It is easy to see from \eqref{eq:3zestb1} that the direct using of
Bourgain space $X^{s,b}$ can only handle one-half derivative, which
was already proved in \cite{MoSaTz1}.

\section{Short-time bilinear estimates}

In this section we prove two bilinear estimates in $F^s$. From the
definition, we divide it into several cases. The first case is
high-low frequency interactions.

\begin{proposition}[high-low]\label{prop:hilow} If $k_3\geq 20$, $|k_2-k_3|\leq 5$, $0\leq k_1\leq k_2-10$, then
\begin{eqnarray}\label{eq:hilow}
\norm{P_{k_3}\partial_x(u_{k_1}v_{k_2})}_{N_{k_3}}\les 2^{k_1/2}
\norm{u_{k_1}}_{F_{k_1}}\norm{v_{k_2}}_{F_{k_2}}.
\end{eqnarray}
\end{proposition}

\begin{proof}
Using the definitions of $N_k$ and \eqref{eq:pXk3}, we obtain that
the left-hand side of \eqref{eq:hilow} is dominated by
\begin{eqnarray}\label{eq:hilowprf1}
&&C\sup_{t_k\in \R}\norm{(\tau-\omega(\xi,\mu)+i2^{k_3})^{-1}\cdot
2^{k_3}1_{I_{k_3}}(\xi) \nonumber\\
&&\quad \cdot \ft[u_{k_1}\eta_0(2^{k_3-2}(t-t_k))]*\ft[v_{k_2}
\eta_0(2^{k_3-2}(t-t_k))]}_{X_k}.
\end{eqnarray}
To prove Proposition \ref{prop:hilow}, it suffices to prove that if
$j_i\geq k_3$ and $f_{k_i,j_i}: \R^3\rightarrow \R_+$ are supported
in ${D}_{k_i,\leq j_i}$ ($D_{\leq 0,\leq j_1}$ if $k_1=0$) for
$i=1,2$, then
\begin{eqnarray}\label{eq:hilowprf2}
2^{k_3}\sum_{j_3\geq k_3}2^{-j_3/2}\norm{1_{{D}_{k_3,\leq j_3}}\cdot
(f_{k_1,j_1}*f_{k_2,j_2})}_{L^2}\les 2^{k_1/2}
2^{(j_1+j_2)/2}\norm{f_{k_1,j_1}}_{L^2}\norm{f_{k_2,j_2}}_{L^2}.
\end{eqnarray}

Indeed, let $f_{k_1}=\ft[u_{k_1}\eta_0(2^{k_3-2}(t-t_k))]$ and
$f_{k_2}=\ft[v_{k_2}\eta_0(2^{k_3-2}(t-t_k))]$. Then from the
definition of $X_k$ we get that \eqref{eq:hilowprf1} is dominated by
\begin{eqnarray}\label{eq:hilowprf3}
\sup_{t_k\in
\R}2^{k_3}\sum_{j_3=0}^{\infty}2^{j_3/2}\sum_{j_1,j_2\geq
k_3}\norm{(2^{j_3}+i2^{k_3})^{-1}1_{{D}_{k_3,j_3}}\cdot
f_{k_1,j_1}*f_{k_2,j_2}}_{L^2},
\end{eqnarray}
where we set
$f_{k_i,j_i}=f_{k_i}(\xi,\mu,\tau)\eta_{j_i}(\tau-\omega(\xi,\mu))$
for $j_i>k_3$ and the remaining part
$f_{k_i,k_3}=f_{k_i}(\xi,\mu,\tau)\eta_{\leq
k_3}(\tau-\omega(\xi,\mu))$, $i=1,2$. For the summation on the terms
$j_3<k_3$ in \eqref{eq:hilowprf3}, we get from the fact
$1_{D_{k_3,j_3}}\leq 1_{{D}_{k_3,\leq j_3}}$ that
\begin{eqnarray}
&&\sup_{t_k\in \R}2^{k_3}\sum_{j_3<k_3}2^{j_3/2}\sum_{j_1,j_2\geq
k_3}\norm{(2^{j_3}+i2^{k_3})^{-1}1_{{D}_{k_3,j_3}}\cdot
f_{k_1,j_1}*f_{k_2,j_2}}_{L^2}\nonumber\\
&&\les \sup_{t_k\in \R}2^{k_3}\sum_{j_1,j_2\geq
k_3}2^{-k_3/2}\norm{1_{\wt{D}_{k_3,k_3}}\cdot
f_{k_1,j_1}*f_{k_2,j_2}}_{L^2}.
\end{eqnarray}
From the fact that $f_{k_i,j_i}$ is supported in $\wt{D}_{k_i,j_i}$
for $i=1,2$ and using \eqref{eq:hilowprf2}, then we get that
\begin{eqnarray*}
&&\sup_{t_k\in \R}2^{k_3}\sum_{j_1,j_2\geq
k_3}2^{-k_3/2}\norm{1_{{D}_{k_3,\leq k_3}}\cdot
f_{k_1,j_1}*f_{k_2,j_2}}_{L^2}\\
 &&\les 2^{k_1/2}\sup_{t_k\in \R}
\sum_{j_1,j_2\geq
k_3}2^{j_1/2}\norm{f_{k_1,j_1}}_{L^2}2^{j_2/2}\norm{f_{k_2,j_2}}_{L^2}.
\end{eqnarray*}
Thus from the definition and using \eqref{eq:pXk2} and
\eqref{eq:pXk3} we obtain \eqref{eq:hilow}, as desired.

To prove \eqref{eq:hilowprf2}, we assume first $k_1\geq 1$. If
$j_{max}\leq k_1+k_2+k_3-20$ on the left-hand side of
\eqref{eq:hilowprf2}, then applying \eqref{eq:3zestb1} we get
\begin{eqnarray*}
&&2^{k_3}\sum_{j_3\geq k_3}2^{-j_3/2}\norm{1_{{D}_{k_3,\leq
j_3}}\cdot
(f_{k_1,j_1}*f_{k_2,j_2})}_{L^2} \nonumber\\
&&\les 2^{k_3}\sum_{j_3\geq
k_3}2^{(j_1+j_2-j_3)/2}2^{k_1/2}2^{-k_3/2}\prod_{i=1}^2\norm{f_{k_i,j_i}}_{L^2}\les
2^{k_1/2}2^{(j_1+j_2)/2}\prod_{i=1}^2\norm{f_{k_i,j_i}}_{L^2}.
\end{eqnarray*}
If $j_{max}\geq k_1+k_2+k_3-20$, then applying
\eqref{eq:3zestbextra} we get
\begin{eqnarray*}
&&2^{k_3}\sum_{j_3\geq k_3}2^{-j_3/2}\norm{1_{{D}_{k_3,\leq
j_3}}\cdot
(f_{k_1,j_1}*f_{k_2,j_2})}_{L^2} \nonumber\\
&&\les 2^{k_3}\sum_{j_3\geq
k_3}2^{(j_1+j_2)/2}2^{-j_3/3}2^{-2k_3/3}2^{k_1/6}\prod_{i=1}^2\norm{f_{k_i,j_i}}_{L^2}\les
2^{k_1/6}2^{(j_1+j_2)/2}\prod_{i=1}^2\norm{f_{k_i,j_i}}_{L^2}.
\end{eqnarray*}
For the case $k_1=0$, we can handle it similarly. Decomposing the
low frequency dyadically and using the same argument as above, then
we see we could sum over the low frequency.
\end{proof}

When the three frequencies are comparable, we have

\begin{proposition}\label{prop:hihihi}
Assume $k_3\geq 20$. If $|k_3-k_2|\leq 5$ and $|k_1-k_2|\leq 5$ then
we have
\begin{eqnarray}\label{eq:hihihi}
\norm{P_{k_3}\partial_x(u_{k_1}v_{k_2})}_{N_{k_3}}\les k_3
2^{-k_3/2}
 \norm{u_{k_1}}_{F_{k_1}}\norm{v_{k_2}}_{F_{k_2}}.
\end{eqnarray}
\end{proposition}
\begin{proof}
As in the proof of Proposition \ref{prop:hilow}, it suffices to
prove that if $j_1,j_2\geq k_3$ and $f_{k_i,j_i}: \R^3\rightarrow
\R_+$ are supported in ${D}_{k_i,\leq j_i}$, $i=1,2$, then
\begin{eqnarray}\label{eq:hihihiprf1}
2^{k_3}\sum_{j_3\geq k_3}2^{-j_3/2}\norm{1_{{D}_{k_3,\leq j_3}}
(f_{k_1,j_1}*f_{k_2,j_2})}_{L^2}\les k_3 2^{-k_3/2}
2^{j_1/2}\norm{f_{k_1,j_1}}_{L^2}2^{j_2/2}\norm{f_{k_2,j_2}}_{L^2}.
\end{eqnarray}
To prove \eqref{eq:hihihiprf1}, clearly we may assume $j_3\leq
10k_3$, otherwise applying \eqref{eq:3zesta} then we have
$2^{-5k_3}$ to spare. Applying \eqref{eq:3zesta} we get that
\begin{eqnarray*}
2^{k_3}\sum_{j_3\geq k_3}2^{-j_3/2}\norm{1_{{D}_{k_3,\leq j_3}}
(f_{k_1,j_1}*f_{k_2,j_2})}_{L^2}\les 2^{(j_1+j_2)/2}
2^{k_3}\sum_{j_3\leq
10k_3}2^{-3k_{3}/2}\prod_{i=1}^2\norm{f_{k_i,j_i}}_{L^2},
\end{eqnarray*}
which gives the bound \eqref{eq:hihihiprf1}, as desired.
\end{proof}

\begin{proposition}[high-high]\label{prop:hhl}
If $k_2\geq 20$, $|k_1-k_2|\leq 5$ and $ 0\leq k_3\leq k_1-10$, then
we have
\begin{eqnarray}\label{eq:hhl}
\norm{P_{k_3}\partial_x(u_{k_1}v_{k_2})}_{N_{k_3}}\les k_22^{-k_3/2}
\norm{u_{k_1}}_{F_{k_1}}\norm{v_{k_2}}_{F_{k_2}}.
\end{eqnarray}
\end{proposition}
\begin{proof}
Let $\beta:\R\rightarrow [0,1]$ be a smooth function supported in
$[-1,1]$ with the property that
\[\sum_{n\in \Z}\beta^2(x-n)\equiv 1, \quad x\in \R.\]
Using the definitions and decomposing the low frequency part, we get
that the left-hand side of \eqref{eq:hhl} is dominated by
\begin{eqnarray*}
&&C\sup_{t_k\in \R}\sum_{k_3'\leq
k_3}\normb{(\tau-\omega(\xi,\eta)+i2^{{k_3'}_+})^{-1}2^{{k_3'}}\chi_{k_3'}(\xi)\sum_{|m|\leq
C2^{(k_2-{k_3'}_+)}}\\
&& \qquad \ft[u_{k_1}\eta_0(2^{{k_3'}_+}(t-t_k))\beta(2^{k_2}(t-t_k)-m)]*\\
&& \qquad
\ft[u_{k_2}\eta_0(2^{{k_3'}_+}(t-t_k))\beta(2^{k_2}(t-t_k)-m)]}_{X_k}.
\end{eqnarray*}

We assume first $k_3=0$. In view of the definitions, \eqref{eq:pXk2}
and \eqref{eq:pXk3}, it suffices to prove that if $j_1,j_2\geq k_2$,
and $f_{k_i,j_i}:\R^3 \ra \R_+$ are supported in ${D}_{k_i,\leq
j_i}$, $i=1,2$, then
\begin{eqnarray}\label{eq:hhlprf1}
&&\sum_{k_3'\leq 0} 2^{k_3'}2^{k_2}\sum_{j_3\geq 0}
2^{-j_3/2}\norm{\chi_{k_3'}(\xi)\eta_{\leq j_3}(\tau-\omega(\xi))(f_{k_1,j_1}*f_{k_2,j_2})}_{L^2}\nonumber\\
&& \les k_22^{j_1/2}\norm{f_{k_1,j_1}}_{L^2}\cdot
2^{j_2/2}\norm{f_{k_2,j_2}}_{L^2}.
\end{eqnarray}
To prove \eqref{eq:hhlprf1}, clearly we may assume $j_3\leq 10k_2$.
Applying \eqref{eq:3zesta} we get that the left-hand side of
\eqref{eq:hhlprf1} is bounded by
\begin{eqnarray*}
2^{k_2}\sum_{k_3'\leq 0}\sum_{j_3\leq 10k_2}
2^{(j_1+j_2)/2}2^{k'_3/2}2^{-k_2}\norm{f_{k_1,j_1}}_{L^2}\cdot
\norm{f_{k_2,j_2}}_{L^2}
\end{eqnarray*}
which gives the bound \eqref{eq:hhlprf1}, as desired.

We assume now $k_3\geq 1$. In view of the definitions,
\eqref{eq:pXk2} and \eqref{eq:pXk3}, it suffices to prove that if
$j_1,j_2\geq k_2$, and $f_{k_i,j_i}:\R^3 \ra \R_+$ are supported in
$\widetilde{D}_{k_i,j_i}$, $i=1,2$, then
\begin{eqnarray}\label{eq:hhl1}
&&2^{k_2}\sum_{j_3\geq k_3}
2^{-j_3/2}\norm{\chi_{k_3}(\xi)\eta_{\leq j_3}(\tau-\omega(\xi))(f_{k_1,j_1}*f_{k_2,j_2})}_{L^2}\nonumber\\
&& \les k_2 2^{-k_3/2}2^{j_1/2}\norm{f_{k_1,j_1}}_{L^2}\cdot
2^{j_2/2}\norm{f_{k_2,j_2}}_{L^2}
\end{eqnarray}
which can be proved in the same way as the case $k_3=0$.
\end{proof}

Finally we consider the low-low interaction. Generally, if
considering the data without special low frequency structure, then
one can always control this interaction. This is different from the
data in the energy space $\E^1$.

\begin{proposition}[low-low]\label{prop:lll}
If $0\leq k_1,k_2,k_3\leq 200$, then
\begin{eqnarray}
\norm{P_{k_3}\partial_x(u_{k_1}v_{k_2})}_{N_{k_3}}\les
\norm{u_{k_1}}_{F_{k_1}}\norm{v_{k_2}}_{F_{k_2}}.
\end{eqnarray}
\end{proposition}
\begin{proof}
From definitions, it suffices to prove
\begin{eqnarray}
&&\norm{(\tau-\omega(\xi,\mu)+i)^{-1}\xi \cdot 1_{I_{\leq
100}}(\xi)\cdot
\ft(u_{k_1})*\ft(v_{k_2})}_{X_{\leq 100}}\nonumber\\
&&\les \norm{\ft(u_{k_1})}_{X_{\leq
100}}\norm{\ft(u_{k_2})}_{X_{\leq 100}}.
\end{eqnarray}
This follows immediately from the definitions, \eqref{eq:3zesta},
\eqref{eq:pXk2} and \eqref{eq:pXk3}.
\end{proof}

As a conclusion to this section we prove the bilinear estimates,
using the dyadic bilinear estimates obtained above.
\begin{proposition}\label{prop:bilinear}
(a) If $s\geq 1$, $T\in (0,1]$, and $u,v\in F^{s}(T)$ then
\begin{eqnarray}
\norm{\partial_x(uv)}_{N^{s}(T)}&\les&
\norm{u}_{F^{s}(T)}\norm{v}_{F^{1}(T)}+\norm{u}_{F^{1}(T)}\norm{v}_{F^{s}(T)}.
\end{eqnarray}

(b)If $T\in (0,1]$, $u\in F^{0}(T)$ and $v\in F^{1}(T)$ then
\begin{eqnarray}
\norm{\partial_x(uv)}_{N^{0}(T)}&\les&
\norm{u}_{F^{0}(T)}\norm{v}_{F^{1}(T)}.
\end{eqnarray}
\end{proposition}
\begin{proof}
Since $P_kP_j=0$ if $k\neq j$ and $k,j\in \Z_+$, then we can fix
extensions $\wt{u}, \wt{v}$ of $u, v$ such that $\norm{P_k(\wt
u)}_{F_k}\leq 2\norm{P_k(u)}_{F_k(T)}$ and $\norm{P_k(\wt
v)}_{F_k}\leq 2\norm{P_k(v)}_{F_k(T)}$ for any $k\in \Z_+$. In view
of definition, we get
\begin{eqnarray*}
\norm{\partial_x(u v
)}_{N^{s}(T)}^2\les\sum_{k_3=0}^{\infty}2^{2sk_3}\norm{P_{k_3}(\partial_x(\wt
u \wt v))}_{N_{k_3}}^2.
\end{eqnarray*}
For $k\in \Z_+$ let $\wt{u}_{k}=P_{k}(\wt u)$ and
$\wt{v}_{k}=P_{k}(\wt v)$, then we get
\begin{eqnarray}\label{eq:bilineares1}
\norm{P_{k_3}(\partial_x(\wt u \wt v))}_{N_{k_3}}\les
\sum_{k_1,k_2\in \Z_+}\norm{P_{k_3}(\partial_x(\wt u_{k_1} \wt
v_{k_2}))}_{N_{k_3}}.
\end{eqnarray}
From symmetry we may assume $k_1\leq k_2$. Dividing the summation on
the right-hand side of \eqref{eq:bilineares1} into several parts, we
get
\begin{eqnarray}
\sum_{k_1,k_2\in \Z_+}\norm{P_{k_3}(\partial_x(\wt u_{k_1} \wt
v_{k_2}))}_{N_{k_3}}&\les& \sum_{i=1}^4
\sum_{{A_i}}\norm{P_{k_3}(\partial_x(\wt u_{k_1} \wt
v_{k_2}))}_{N_{k_3}}
\end{eqnarray}
where we denote
\begin{eqnarray*}
&&A_1=\{k_1\leq k_2: |k_2-k_3|\leq 5, k_1\leq k_2-10, \mbox{ and }
k_2\geq
20\};\\
&&A_2=\{k_1\leq k_2: |k_2-k_3|\leq 5, |k_1-k_2|\leq 10, \mbox{ and }
k_2\geq
20\};\\
&&A_3=\{k_1\leq k_2: k_3\leq k_2-10, |k_1-k_2|\leq 5, \mbox{ and }
k_1 \geq 20\}.\\
&&A_4=\{k_1\leq k_2: k_1,k_2,k_3\leq 200\}.
\end{eqnarray*}

For part (a), it suffices to prove that for $i=1,2,3,4$ then
\begin{eqnarray}
\normo{2^{sk_3}\sum_{{A_i}}\norm{P_{k_3}(\partial_x(\wt u_{k_1} \wt
v_{k_2}))}_{N_{k_3}}}_{l^2_{k_3}}\les \norm{\wt u}_{F^{s}}\norm{\wt
v}_{F^{1}}+\norm{\wt u}_{F^{1}}\norm{\wt v}_{F^{s}},
\end{eqnarray}
which follows from Proposition \ref{prop:hilow}-\ref{prop:lll}. For
part (b), it suffices to prove
\begin{eqnarray}\label{eq:bib}
\normo{\sum_{{k_1,k_2\in \Z_+}}\norm{P_{k_3}(\partial_x(\wt u_{k_1}
\wt v_{k_2}))}_{N_{k_3}}}_{l^2_{k_3}}\les \norm{\wt
v}_{F^{0}}\norm{\wt u}_{F^{1}}.
\end{eqnarray}
which follows in the same ways.
\end{proof}

\section{Proof of Theorem \ref{thmmain}}

In this section we devote to prove Theorem \ref{thmmain}. The main
ingredients are energy estimates which are proved in the next
section and short-time bilinear estimates obtained in the last
section. The method is due to Ionescu, Kenig and Tataru \cite{IKT}.
We will also need the local well-posedness for more regular
solution.
\begin{theorem}[Theorem 2, \cite{MoSaTz4}]\label{thmmst}
The KP-I initial-value problem \eqref{eq:kpI} is locally well-posed
in $H^{s,0}$ for $s>3/2$.
\end{theorem}

Theorem \ref{thmmst} is proved by refined energy methods. The length
of existence interval $T$ is determined by $\norm{\phi}_{H^{s,0}}$.

\begin{proposition}\label{pFstoHs}
Let $s\geq 0$, $T\in (0,1]$, and $u\in F^{s}(T)$, then
\begin{equation}
\sup_{t\in [-T,T]}\norm{u(t)}_{{H}^{s,0}}\les\ \norm{u}_{F^{s}(T)}.
\end{equation}
\end{proposition}
\begin{proof}
 In view of the definitions, it suffices to prove that if $k\in
\Z_+$, $t_k\in [-1,1]$, and $\wt u_k \in F_k$ then
\begin{equation}
\norm{\ft[\wt u_k(t_k)]}_{L_{\xi,\mu}^2}\les \norm{\ft[\wt u_k\cdot
\eta_0(2^{k}(t-t_k))]}_{X_k}.
\end{equation}
Let $f_k=\ft[\wt u_k\cdot \eta_0(2^{k}(t-t_k))]$, then
\[\ft[\wt u_k(t_k)](\xi,\mu)=c\int_\R f_k(\xi,\mu,\tau)e^{it_k\tau}d\tau.\]
From the definition of $X_k$, we get that
\[ \norm{\ft[\wt
u_k(t_k)]}_{L^2}\les \normo{\int_\R
|f_k(\xi,\mu,\tau)|d\tau}_{L^2}\les \norm{f_k}_{X_k},\] which
completes the proof of the proposition.
\end{proof}

\begin{proposition}\label{prop:linear}
Assume $T\in (0,1]$, $u,v\in C([-T,T]:H^{\infty,0})$ and
\begin{eqnarray}\label{eq:lKPI}
u_t+\partial_x^3 u-\partial_x^{-1}\partial_y^2 u=v \mbox{ on } \R^2
\times (-T,T).
\end{eqnarray}
Then for any $s\geq 0$,
\begin{equation}\label{eq:linear}
\norm{u}_{F^{s}(T)}\les \ \norm{u}_{E^{s}(T)}+\norm{v}_{N^{s}(T)}.
\end{equation}
\end{proposition}
\begin{proof}
In view of the definitions, we see that the square of the right-hand
side of \eqref{eq:linear} is equivalent to
\begin{eqnarray*}
&&\norm{P_{\leq 0}(u(0))}_{L^2}^2+\norm{P_{\leq
0}(v)}_{N_k(T)}^2\\
&&+\sum_{k\geq 1}\big(\sup_{t_k\in
[-T,T]}2^{2sk}\norm{P_k(u(t_k))}_{L^2}^2
+2^{2sk}\norm{P_k(v)}_{N_k(T)}^2\big).
\end{eqnarray*}
Thus, from definitions, it suffices to prove that if $k\in \Z_+$ and
$u,v \in C([-T,T]:H^{\infty,0})$ solve \eqref{eq:lKPI}, then
\begin{eqnarray}\label{eq:retardlinear}
\left\{\begin{array}{l} \norm{P_{\leq 0}(u)}_{F_0(T)}\les
\norm{P_{\leq 0}(u(0))}_{L^2}+\norm{P_{\leq 0}(v)}_{N_0(T)};\\
\norm{P_k(u)}_{F_k(T)}\les \sup_{t_k\in
[-T,T]}\norm{P_k(u(t_k))}_{L^2}+\norm{P_k(v)}_{N_k(T)} \mbox{ if }
k\geq 1.
\end{array}
\right.
\end{eqnarray}

We only prove the second inequality in \eqref{eq:retardlinear},
since the first one can be treated in the same ways. Fix $k\geq 1$
and let $\wt v$ denote an extension of $P_k(v)$ such that $\norm{\wt
v}_{N_k}\leq C\norm{v}_{N_k(T)}$. In view of \eqref{eq:Sk}, we may
assume that $\wt v$ is supported in $\R\times
[-T-2^{-k-10},T+2^{-k-10}]$. For $t\geq T$ we define
\[\wt u (t)=\eta_0(2^{k+5}(t-T))\big[W(t-T)P_k(u(T))+\int_T^tW(t-s)(P_k(\wt v(s)))ds \big].\]
For $t\leq -T$ we define
\[\wt u (t)=\eta_0(2^{k+5}(t+T))\big[W(t+T)P_k(u(-T))+\int_{-T}^tW(t-s)(P_k(\wt v(s)))ds \big].\]
For $t\in [-T,T]$ we define $\wt u(t)=u(t)$. It is clear that $\wt
u$ is an extension of u and we get from \eqref{eq:Sk} that
\begin{eqnarray}\label{eq:extu}
\norm{u}_{F_k(T)}\les \sup_{t_k\in [-T,T]}\norm{\ft[\wt u \cdot
\eta_0(2^{k}(t-t_k))]}_{X_k}.
\end{eqnarray}

Now we prove the second inequality in \eqref{eq:retardlinear}. In
view of the definitions, \eqref{eq:extu} and \eqref{eq:pXk3}, it
suffices to prove that if $\phi_k \in L^2$ with $\widehat{\phi_k}$
supported in $I_k$, and $v_k\in N_k$ then
\begin{eqnarray}
\norm{\ft[u_k\cdot \eta_0(2^{k}t)]}_{X_k}\les
\norm{\phi_k}_{L^2}+\norm{(\tau-\omega(\xi,\mu)+i2^{k})^{-1}\cdot
\ft(v_k)}_{X_k},
\end{eqnarray}
where
\begin{equation}
u_k(t)=W(t)(\phi_k)+\int_0^tW(t-s)(v_k(s))ds.
\end{equation}
Straightforward computations show that
\begin{eqnarray*}
&&\ft[u_k\cdot \eta_0(2^{k}t)](\xi,\tau)=\widehat{\phi_k}(\xi)\cdot
2^{-k}\widehat{\eta_0}(2^{-k}(\tau-\omega(\xi,\mu)))\\
&&+C\int_\R \ft(v_k)(\xi,\tau')\cdot
\frac{\widehat{\eta_0}(2^{-k}(\tau-\tau'))-\widehat{\eta_0}(2^{-k}(\tau-\omega(\xi,\mu)))}{2^{k}(\tau'-\omega(\xi,\mu))}d\tau'.
\end{eqnarray*}
We observe now that
\begin{eqnarray*}
&&\aabs{\frac{\widehat{\eta_0}(2^{-k}(\tau-\tau'))-\widehat{\eta_0}(2^{-k}(\tau-\omega(\xi,\mu)))}{2^{k}(\tau'-\omega(\xi,\mu))}\cdot
(\tau'-\omega(\xi,\mu)+i2^{k})}\\
&&\les \
2^{-k}(1+2^{-k}|\tau-\tau'|)^{-4}+2^{-k}(1+2^{-k}|\tau-\omega(\xi,\mu)|)^{-4}.
\end{eqnarray*}
Using \eqref{eq:pXk1} and \eqref{eq:pXk2}, we complete the proof of
the proposition.
\end{proof}

Now we turn to prove Theorem \ref{thmmain}. We note that KP-I
equation \eqref{eq:kpI} is invariant under the following scaling
transform
\begin{eqnarray}\label{eq:scaling}
u(x,y,t)\rightarrow u_\lambda(x,y,t)=\lambda^2 u(\lambda x,
\lambda^2 y, \lambda^3 t).
\end{eqnarray}
Thus we see $\dot{H}^{s_1,s_2}$ is the critical space if
$s_1+2s_2=-1/2$ in the sense of scaling. To prove Theorem
\ref{thmmain} (a), by the scaling we may assume that
\begin{equation}\label{eq:smalldata}
\norm{u_0}_{H^{1,0}}\leq \epsilon_0 \ll 1.
\end{equation}
We only need to construct the solution on the time interval
$[-1,1]$. In view of Theorem \ref{thmmst}, it suffices to prove that
if $T\in (0,1]$ and $u\in C([-T,T]:H^{\infty,0})$ is a solution of
\eqref{eq:kpI} with $\norm{u_0}_{H^{1,0}}\leq \epsilon\ll 1$ then
\begin{eqnarray}\label{eq:H2est}
\sup_{t\in [-T,T]}\norm{u(t)}_{H^{2,0}}\les \norm{u_0}_{H^{2,0}}.
\end{eqnarray}

It follows from Proposition \ref{prop:linear}, Proposition
\ref{prop:bilinear} and the energy estimate Proposition
\ref{prop:energy} that for any $T'\in [0,T]$ we have
\begin{eqnarray}\label{eq:Fsest}
\left \{
\begin{array}{l}
\norm{u}_{F^{1}(T')}\les \norm{u}_{E^{1}(T')}+\norm{\partial_x(u^2)}_{N^{1}(T')};\\
\norm{\partial_x(u^2)}_{N^{1}(T')}\les \norm{u}_{F^{1}(T')}^2;\\
\norm{u}_{E^{1}(T')}^2\les
\norm{\phi}_{{H}^{1,0}}^2+\norm{u}_{F^{1}(T')}^3.
\end{array}
\right.
\end{eqnarray}
We denote
$X(T')=\norm{u}_{E^1(T')}+\norm{\partial_x(u^2)}_{N^{1}(T')}$. Then
by a similar argument as in the proof of Lemma 4.2 in \cite{IKT}, we
know $X(T')$ is continuous and satisfies
\[\lim_{T'\rightarrow 0}X(T')\les \norm{u_0}_{H^{1,0}}. \]
On the other hand, we get from \eqref{eq:Fsest} that
\begin{eqnarray*}
X(T')^2\les \norm{u_0}_{H^{1,0}}^2+X(T')^3+X(T')^4.
\end{eqnarray*}
If $\epsilon_0$ is sufficiently small, then we can get from
\eqref{eq:smalldata}, the continuity of $X(T)$ and the standard
bootstrap that $X(T')\les \norm{u_0}_{H^{1,0}}$ and therefore we
obtain
\begin{eqnarray}\label{eq:smallFs}
\norm{u}_{F^{1,0}(T)}\les \norm{u_0}_{H^{1,0}}.
\end{eqnarray}

For $\sigma\geq 1$ we obtain from Proposition \ref{prop:linear},
Proposition \ref{prop:bilinear} (a) and the energy estimate
Proposition \ref{prop:energy} that for any $T'\in [0,T]$ we have
\begin{eqnarray}\label{eq:Fsest2}
\left \{
\begin{array}{l}
\norm{u}_{F^{\sigma}(T')}\les \norm{u}_{E^{\sigma}(T')}+\norm{\partial_x(u^2)}_{N^{\sigma}(T')};\\
\norm{\partial_x(u^2)}_{N^{\sigma}(T')}\les \norm{u}_{F^{\sigma}(T')}\norm{u}_{F^{1}(T')};\\
\norm{u}_{E^{\sigma}(T')}^2\les
\norm{\phi}_{{H}^{\sigma,0}}^2+\norm{u}_{F^{1}(T')}\norm{u}_{F^{\sigma}(T')}^2.
\end{array}
\right.
\end{eqnarray}
Then from \eqref{eq:smallFs} we get $\norm{u}_{F^1(T)}\ll 1$ and
hence
\begin{eqnarray}
\norm{u}_{F^\sigma(T)}\les \norm{u_0}_{H^{\sigma,0}},
\end{eqnarray}
which in particularly implies \eqref{eq:H2est} as desired. We
complete the proof of part (a).

We prove now Theorem \ref{thmmain} (b), using the Bona-Smith
argument \cite{BonaSmith} as in \cite{IKT}. Fixing $u_0\in H^{1,0}$,
then we choose $\{\phi_n\}\subset H^{\infty,0}$ such that
$\lim_{n\rightarrow \infty} \phi_n=u_0$ in $H^{1,0}$. It suffices to
prove the sequence $S_T^\infty(\phi_n)$ is a Cauchy sequence in
$C([-T,T]:H^{1,0})$. From the definition it suffices to prove that
for any $\delta>0$ there is $M_\delta$ such that
\[\sup_{t\in [-T,T]}\norm{S_T^\infty(\phi_m)-S_T^\infty(\phi_n)}_{H^{1,0}}\leq \delta, \quad \forall\ m,n\geq M_\delta.\]
For $K\in \Z_+$ let $\phi_n^K=P_{\leq K}\phi_n$. Since
$\phi_n^K\rightarrow u_0^K$ in $H^{2,0}$, then we see for any fixed
$K$ there is $M_{\delta,K}$ such that
\[\sup_{t\in [-T,T]}\norm{S_T^\infty(\phi_m^K)-S_T^\infty(\phi_n^K)}_{H^{1,0}}\leq \delta/2, \quad \forall\ m,n\geq M_{\delta,K}.\]
On the other hand, we get from Proposition \ref{prop:energydiff} and
Lemma \ref{pFstoHs} that
\begin{eqnarray*}
\sup_{t\in[-T,T]}\norm{S_T^\infty(\phi_n)-S_T^\infty(\phi_n^K)}_{H^{1,0}}
&\les&\norm{S_T^\infty(\phi_n)-S_T^\infty(u_n^K)}_{F^1(T)}\\
&\les&\norm{\phi_n-\phi_n^K}_{H^{1,0}}+\norm{\phi_n^K}_{H^{2,0}}\norm{\phi_n-\phi_n^K}_{L^2}\\
&\les& \norm{\phi-\phi_n}_{H^{1,0}}+\norm{\phi-\phi^K}_{H^{1,0}}.
\end{eqnarray*}
Thus we obtain that for any $\delta>0$ there are $K$ and $M_\delta$
such that
\[\sup_{t\in
[-T,T]}\norm{S_T^\infty(\phi_n)-S_T^\infty(\phi_n^K)}_{H^{1,0}}\leq
\delta/2, \quad \forall\ n\geq M_\delta.\] Therefore, we complete
the proof of part (b) of Theorem \ref{thmmain}.

\section{Energy Estimates}

In this section we prove the energy estimates, by following the
ideas in \cite{IKT}. We introduce a new Littlewood-Paley
decomposition with smooth symbols. With
\[\chi_k(\xi)=\eta_0(\xi/2^k)-\eta_0(\xi/{2^{k-1}}), \quad k\in \Z,\]
Let $\widetilde{P}_k$ denote the operator on $L^2(\R)$ defined by
the Fourier multiplier $\chi_k(\xi)$ and similarly define the
operator $\wt{P}_{\leq k}$. Assume that $u,v\in C([-T,T];L^2)$ and
\begin{equation*}
\begin{cases}
\partial_tu+\partial_x^3u-\partial_x^{-1}\partial_y^2u=v \text{ on }\mathbb{R}^2_{x,y}\times\mathbb{R}_t;\\
u(0)=\phi,
\end{cases}
\end{equation*}
Then we multiply by $u$ and integrate to conclude that
\begin{equation}\label{eq:L2esti}
\sup\limits_{|t_k|\leq T}\norm{u(t_k)}_{L^2}^2\leq
\norm{\phi}_{L^2}^2+\sup\limits_{|t_k|\leq
T}\aabs{\int_{\R\times[0,t_k]}u\cdot v \ dxdydt}.
\end{equation}
In applications we usually take $v=\partial_x(u^2)$. This particular
term has a cancelation that we need to exploit.

\begin{lemma}\label{lem:3linear}
(a) Assume $T\in (0,1]$, $k_1,k_2,k_3 \in \Z_+$, and $u_i\in
F_{k_i}(T), i=1,2,3$. Then if $k_{min}\leq k_{max}-5$, we have
\begin{eqnarray}\label{eq:3lineara}
\aabs{\int_{\R^2\times [0,T]}u_1u_2u_3\ dxdydt}\les 2^{-k_{min}/2}
\prod_{i=1}^3 \norm{u_i}_{F_{k_i}(T)}.\label{eq:tri1}
\end{eqnarray}

(b) Assume $T\in (0,1]$, $ k\in \Z_+$, $0\leq k_1\leq k-10$, $u\in
F_k(T)$, and $v\in F_{k_1}(T)$. Then we have
\begin{eqnarray}\label{eq:3linearb}
\aabs{\int_{\R^2\times
[0,T]}\widetilde{P}_k(u)\widetilde{P}_k(\partial_x u \cdot
\widetilde{P}_{k_1}(v))dxdydt}\les 2^{k_{1}/2}
\norm{v}_{F_{k_1}(T)}\sum_{|k'-k|\leq
10}\norm{\wt{P}_{k'}(u)}_{F_{k'}(T)}^2.
\end{eqnarray}
If $k_1=0$, \eqref{eq:3linearb} also holds if $\widetilde{P}_{k_1}$
is replaced by $\widetilde{P}_{\leq 0}$.
\end{lemma}

\begin{proof}
For part (a), from symmetry we may assume $k_1\leq k_2\leq k_3$.  We
fix extension $\wt{u}_i \in F_{k_i}$ such that
$\norm{\wt{u}_i}_{F_{k_i}}\leq 2\norm{u_i}_{F_{k_i}(T)}$, $i=1,2,3$.
If $k_{3}\leq 10$, then
\begin{eqnarray*}
\aabs{\int_{\R^2\times [0,T]}u_1u_2u_3\ dxdydt}\les
\norm{u_3\eta_0(t)}_{L^2}\norm{u_1\eta_0(t)}_{L^4}\norm{u_2\eta_0(t)}_{L^4}
\end{eqnarray*}
It is easy to see that $\norm{u_3\eta_0(t)}_{L^2}\les
\norm{u_3}_{X_{k_3}}$. On the other hand, we also know for all $u\in
X_k$ then
\begin{eqnarray}\label{eq:Xkembedding}
\norm{\ft^{-1}(u)}_{L^4_{x,y,t}}\les \norm{u}_{X_k}
\end{eqnarray}
which suffices to prove (a) in this case in view of \eqref{eq:pXk3}.
To prove \eqref{eq:Xkembedding}, we have
\begin{eqnarray*}
\ft^{-1}(u)&=&\int_{\R^3} u(\xi,\mu,\tau)e^{ix\xi+iy\mu+it\tau}d\xi
d\mu d\tau\\
&=&\int_{\R^3}
u(\xi,\mu,\tau+\dr(\xi,\mu))e^{ix\xi+iy\mu+it\tau}e^{it\dr(\xi,\mu)}d\xi
d\mu d\tau.
\end{eqnarray*}
Using the Strichartz estimate $\norm{W(t)\phi}_{L^4}\les
\norm{\phi}_{L^2}$ in \cite{ArSa}, we immediately get
\eqref{eq:Xkembedding}.

We consider now $k_3\geq 10$. In order for the integral to be
nontrivial we must have $|k_2-k_3|\leq 4$. If $k_1\geq 1$, this is
proved in \cite{IKT}. We only need to prove the case $k_1=0$. Let
$\gamma:\R \rightarrow [0,1]$ denote a smooth function supported in
$[-1,1]$ with the property that
\[\sum_{n\in \Z} \gamma^3(x-n)\equiv 1, \quad x\in \R.\]
The left-hand side of \eqref{eq:tri1} is dominated by
\begin{eqnarray}\label{eq:3linearaprf1}
\sum_{|n|\leq C2^{k_3}}\bigg|\int_{\R\times \R} \big(
\gamma(2^{k_3}t-n)\wt{u}_1\big)\big(\gamma(2^{k_3}t-n)\wt{u}_2
\big)\big(\gamma(2^{k_3}t-n)1_{[0,T]}(t)\wt{u}_3\big) dxdydt\bigg|.
\end{eqnarray}
We observe first that
\[|A|=|\{n: \gamma(2^{k_3}t-n)1_{[0,T]}(t) \mbox{ nonzero and } \ne \gamma(2^{k_3}t-n)\}|\leq 4.\]

For the summation of $n\in A^c$ on the left-hand side of
\eqref{eq:3linearaprf1}, as was explained in the proof of
Proposition \ref{prop:hilow}, for \eqref{eq:3lineara} it suffices to
prove that if $f_{k_i,j_i}$ are $L^2$ functions supported in
${D}_{k_i,\leq j_i}$ for $i=2,3$ and $f_{0,j_1}$ is a $L^2$ function
supported in ${D}_{\leq 0,\leq j_1}$, then
\begin{eqnarray}\label{eq:3linearaprf2}
2^{k_3}\sum_{j_1,j_2,j_3\geq
k_3}|J(f_{k_1,j_1},f_{k_2,j_2},f_{k_3,j_3})|\les \sum_{j_i\geq 0}
\prod_{i=1}^3 2^{j_i/2}\norm{f_{k_i,j_i}}_{2}.
\end{eqnarray}
Decomposing the low frequency $f_{0,j_1}=\sum_{k'\leq
0}f_{k',j_1}=\sum_{k'\leq 0}\chi_{k'}(\xi)f_{0,j_1}$, then we see
\[2^{k_3}\sum_{j_1,j_2,j_3\geq
k_3}|J(f_{k_1,j_1},f_{k_2,j_2},f_{k_3,j_3})|\leq 2^{k_3}\sum_{k'\leq
0}\sum_{j_1,j_2,j_3\geq
k_3}|J(f_{k',j_1},f_{k_2,j_2},f_{k_3,j_3})|\] If $j_{max}\leq
k'+k_2+k_3-10$, then using \eqref{eq:3zesta}, we get that
\[2^{k_3}\sum_{k'\leq 0}\sum_{j_1,j_2,j_3\geq
k_3}2^{-k_2/2}2^{k'/2}2^{j_1/2}2^{j_2/2}\prod_{i=1}^3
\norm{f_{k_i,j_i}}_{2}\les \sum_{j_i\geq 0} \prod_{i=1}^3
2^{j_i/2}\norm{f_{k_i,j_i}}_{2}.\] On the other hand, if
$j_{max}\geq k'+k_2+k_3-10$, then using \eqref{eq:3zestbextra}, we
get that
\[2^{k_3}\sum_{k'\leq
0}\sum_{j_1,j_2,j_3\geq k_3}2^{(j_1+j_2)/2}2^{j_{3}/6}
2^{-2k_3/3}2^{k'/6}\prod_{i=1}^3 \norm{f_{k_i,j_i}}_{2}\les
\sum_{j_i\geq 0} \prod_{i=1}^3 2^{j_i/2}\norm{f_{k_i,j_i}}_{2}.\]
Thus the summation of $n\in A^c$ is under control.

For the summation of $n\in A$, we observe that if $I\subset \R$ is
an interval, $k\in \Z_+$, $f_k\in X_k$, and $f_k^I=\ft(1_I(t)\cdot
\ft^{-1}(f_k))$ then
\begin{eqnarray*}
\sup_{j\in \Z_+}2^{j/2}\norm{\eta_j(\tau-\omega(\xi))\cdot
f_k^I}_{L^2}\les \norm{f_k}_{X_k}.
\end{eqnarray*}
Thus using \eqref{eq:3zesta} and as for summation on $A^c$, we get
the bound as desired.

For part(b), \eqref{eq:3linearb} is proved in \cite{IKT} for all
$k_1\in \Z$. For $k_1=0$ and $\wt{P}_{\leq 0}$ in
\eqref{eq:3linearb}, we could decompose the low frequency and then
apply \eqref{eq:3linearb}.
\end{proof}

\begin{proposition}\label{prop:energy}
Assume that $T\in (0,1]$ and $u\in C([-T,T]:H^{\infty,0})$ is a
solution to \eqref{eq:kpI} on $\R\times(-T,T)$. Then for $s\geq 1$
we have
\begin{eqnarray}\label{eq:energy}
\norm{u}_{E^s(T)}^2\les
\norm{u_0}_{H^{s,0}}^2+\norm{u}_{F^{1}(T)}\norm{u}_{F^{s}(T)}^2.
\end{eqnarray}
\end{proposition}

\begin{proof}
From definition we have
\begin{eqnarray*}
\norm{u}_{E^s(T)}^2-\norm{P_{\leq 0}(u_0)}_{L^2}^2\les \sum_{k\geq
1}\sup_{t_k\in [-T,T]}2^{2sk}\norm{\wt{P}_k(u(t_k))}_{L^2}^2.
\end{eqnarray*}
Then we can get from \eqref{eq:L2esti} that
\begin{eqnarray}\label{eq:energyeq}
2^{2sk}\norm{\wt{P}_k(u(t_k))}_{L^2}^2-2^{2sk}\norm{\wt{P}_k(u_0)}_{L^2}^2\les
2^{2sk}\left|\int_{\R\times [0,t_k]}\wt{P}_k(u)\wt{P}_k(u\cdot
\partial_x u)dxdt\right|.
\end{eqnarray}
It is easy to see that the right-hand side of \eqref{eq:energyeq} is
dominated by
\begin{eqnarray}\label{eq:trigoal}
&&C2^{2sk}\sum_{0\leq k_1\leq k-10}\left|\int_{\R\times
[0,t_k]}\wt{P}_k(u)\wt{P}_k(\wt{P}_{k_1}u\cdot\partial_x u)dxdt\right|\nonumber\\
&&+C2^{2sk}\sum_{k_1\geq k-9,k_2\in \Z_+}\left|\int_{\R\times
[0,t_k]}\wt{P}_k^2(u)\wt{P}_{k_1}(u)\cdot\partial_x \wt{P}_{k_2}(
u)dxdt\right|.
\end{eqnarray}
For the first term in \eqref{eq:trigoal},  using \eqref{eq:3linearb}
then we get that it is bounded by
\begin{eqnarray*}
&&C2^{2sk}\sum_{k_1\leq k-10}2^{k_1/2}
\norm{u}_{F_{k_1}(T)}\sum_{|k'-k|\leq
10}\norm{\wt{P}_{k'}(u)}_{F_{k'}(T)}^2\\
&&\les \norm{u}_{F^{\rev{2}+}(T)}2^{2sk}\sum_{|k'-k|\leq
10}\norm{u}_{F_{k'}(T)}^2
\end{eqnarray*}
which implies that the summation of the first term is bounded by
$\norm{u}_{F^{\rev{2}+}(T)}\norm{u}_{F^{s}(T)}^2$ as desired.

For the second term in \eqref{eq:trigoal}, using \eqref{eq:tri1} we
get that it is bounded by
\begin{eqnarray*}
&&C2^{2sk}\sum_{|k_1-k|\leq 10,k_2\leq k+10}2^{k_2/2}\norm{\wt{P}_k(u)}_{F_k(T)}\norm{\wt{P}_{k_1}(u)}_{F_{k_1}(T)}\norm{\wt{P}_{k_2}(u)}_{F_{k_2}(T)}\\
&&+C2^{2sk}\sum_{|k_1-k_2|\leq 10,k_1\geq
k+10}2^{k_2-k/2}\norm{\wt{P}_k(u)}_{F_k(T)}\norm{\wt{P}_{k_1}(u)}_{F_{k_1}(T)}\norm{\wt{P}_{k_2}(u)}_{F_{k_2}(T)}.
\end{eqnarray*}
Then it is easy to see that summation over $k_3\geq 1$ s bounded by
$\norm{u}_{F^{1}(T)}\norm{u}_{F^{s}(T)}^2$. We complete the proof of
the proposition.
\end{proof}

\begin{proposition}\label{prop:energydiff}

Let $u_1,u_2 \in F^{1}(1)$ be solutions to \eqref{eq:kpI} with
initial data $\phi_1,\phi_2 \in H^{\infty,0}$ satisfying
\begin{eqnarray}\label{eq:energydiffsmall}
\norm{u_1}_{F^{1}(1)}+\norm{u_2}_{F^{1}(1)}\leq \epsilon_0\ll 1.
\end{eqnarray}
Then we have
\begin{eqnarray}\label{eq:L2conti}
\norm{u_1-u_2}_{F^0(1)}\les \norm{\phi_1-\phi_2}_{L^2},
\end{eqnarray}
and
\begin{eqnarray}\label{eq:H1conti}
\norm{u_1-u_2}_{F^{1}(1)}\les
\norm{\phi_1-\phi_2}_{H^{1,0}}+\norm{\phi_1}_{H^{2,0}}\norm{\phi_1-\phi_2}_{L^2}.
\end{eqnarray}
\end{proposition}

\begin{proof}
We prove first \eqref{eq:L2conti}. Let $v=u_2-u_1$, then $v$ solves
the equation
\begin{eqnarray}\label{eq:energydiffprf1}
\left\{
\begin{array}{l}
\partial_t v+\partial_x^3 v-\partial_x^{-1}\partial_y^2 v=-\partial_x[v(u_1+u_2)/2];\\
v(0)=\phi=\phi_2-\phi_1.
\end{array}
\right.
\end{eqnarray}
Then from Proposition \ref{prop:linear} and Proposition
\ref{prop:bilinear} (b) we obtain
\begin{eqnarray}\label{eq:energydiffprf1}
\left\{
\begin{array}{l}
\norm{v}_{F^0(1)}\les \norm{v}_{E^0(1)}+\norm{\partial_x[v(u_1+u_2)/2]}_{N^0(1)};\\
\norm{\partial_x[v(u_1+u_2)/2]}_{N^0(1)}\les
\norm{v}_{F^0(1)}(\norm{u_1}_{F^{1}(1)}+\norm{u_2}_{F^{1}(1)}).
\end{array}
\right.
\end{eqnarray}
We derive an estimate on $\norm{v}_{E^0(1)}$. As in the proof of
Proposition \ref{prop:energy}, we get from \eqref{eq:L2esti} that
\begin{eqnarray}\label{eq:energydiffprf2}
\norm{v}_{E^0(1)}^2-\norm{\phi}_{L^2}^2&\les&
\sum_{k\geq1}\left|\int_{\R\times
[0,t_k]}\wt{P}_k(v)\wt{P}_k(\partial_x £¨v£©\cdot
(u_1+u_2))dxdt\right|\nonumber\\
&& +\sum_{k\geq1}\left|\int_{\R\times
[0,t_k]}\wt{P}_k(v)\wt{P}_k(v\cdot
\partial_x (u_1+u_2))dxdt\right|.
\end{eqnarray}
For the first term on right-hand side of \eqref{eq:energydiffprf2},
using Lemma \ref{lem:3linear} we can bound it by
\begin{eqnarray*}
&&C\sum_{k\geq 1}\sum_{k_1\leq k-10}\left|\int_{\R\times
[0,t_k]}\wt{P}_k(v)\wt{P}_k(\partial_x v\cdot\wt{P}_{k_1}(u_1+u_2))dxdt\right|\\
&&+C\sum_{k\geq 1}\sum_{k_1\geq k-9,k_2\in \Z_+}\left|\int_{\R\times
[0,t_k]}\wt{P}_k^2(v)\partial_x\wt{P}_{k_2}(v)\cdot \wt{P}_{k_1}(
u_1+u_2)dxdt\right|\\
&&\les\norm{v}_{F^0(1)}^2(\norm{u_1}_{F^{1}(1)}+\norm{u_2}_{F^{1}(1)}),
\end{eqnarray*}
The second term on right-hand side of \eqref{eq:energydiffprf2} is
dominated by
\begin{eqnarray*}
&&\sum_{k\geq1}\sum_{k_1,k_2\in \Z_+}\left|\int_{\R\times
[0,t_k]}\wt{P}_k^2(v)\wt{P}_{k_1}(v)\cdot
\partial_x \wt{P}_{k_2}(u_1+u_2)dxdt\right|\\
&&\les
\norm{v}_{F^0(1)}^2(\norm{u_1}_{F^{1}(1)}+\norm{u_2}_{F^{1}(1)}).
\end{eqnarray*}
Therefore, we obtain the following estimate
\begin{eqnarray}
\norm{v}_{E^0(1)}^2\les
\norm{\phi}_{L^2}^2+\norm{v}_{F^0(1)}^2(\norm{u_1}_{F^{1}(1)}+\norm{u_2}_{F^{1}(1)}),
\end{eqnarray}
which combined with \eqref{eq:energydiffprf2} implies
\eqref{eq:L2conti} in view of \eqref{eq:energydiffsmall}.

We prove now \eqref{eq:H1conti}. From Proposition \ref{prop:linear}
and \ref{prop:bilinear} we obtain
\begin{eqnarray}\label{eq:energydiff2prf1}
\left\{
\begin{array}{l}
\norm{v}_{F^{1}(1)}\les \norm{v}_{E^{1}(1)}+\norm{\partial_x[v(u_1+u_2)/2]}_{N^{1}(1)};\\
\norm{\partial_x[v(u_1+u_2)/2]}_{N^{1}(1)}\les
\norm{v}_{F^{1}(1)}(\norm{u_1}_{F^{1}(1)}+\norm{u_2}_{F^{1}(1)}).
\end{array}
\right.
\end{eqnarray}
Since $\norm{P_{\leq 0}(v)}_{E^{1}(1)}=\norm{P_{\leq
0}(\phi)}_{L^2}$, it follows from \eqref{eq:energydiffsmall} that
\begin{eqnarray}\label{eq:energyv}
\norm{v}_{F^{1}(1)}\les \norm{P_{\geq
1}(v)}_{E^{1}(1)}+\norm{\phi}_{H^{1,0}}.
\end{eqnarray}
To bound $\norm{P_{\geq 1}(v)}_{E^{1}(1)}$, we observe that
\[\norm{P_{\geq 1}(v)}_{E^{1}(1)}=\norm{P_{\geq
1}(\partial_xv)}_{E^0(1)},\] We write the equation for $U={P}_{\geq
-10}(\partial_x v)$ in the form
\begin{eqnarray}\label{eq:energyU}
\left\{
\begin{array}{l}
\partial_t U+\partial_x^3 U-\partial_x^{-1}\partial_y^2 U=P_{\geq -10}(-u_2\cdot\partial_x U)+P_{\geq -10}(G);\\
U(0)=P_{\geq -10}(\Lambda^{\sigma}\phi),
\end{array}
\right.
\end{eqnarray}
where
\begin{eqnarray*}
G&=&-P_{\geq -10}(u_2)\cdot \partial_x^2 P_{\leq -11}(v)-P_{\leq
-11}(u_2)\cdot \partial_x^2 P_{\leq
-11}(v)\\
&&-\partial_x v\partial_x(u_1+u_2)-v\cdot
\partial_x^2 u_1.
\end{eqnarray*}

It follows from \eqref{eq:L2esti} and \eqref{eq:energyU} that
\begin{eqnarray*}
\norm{U}_{E^0(1)}^2-\norm{\phi}_{H^{1}}^2&\les& \sum_{k\geq 1}
\left|\int_{\R\times [0,t_k]}\wt{P}_k(U)\wt{P}_{k}(u_2\cdot
\partial_x U)dxdt\right|\\
&& +\sum_{k\geq 1} \left|\int_{\R\times [0,t_k]}\wt{P}_k^2(U)P_{\geq
-10}(u_2)\cdot \partial_x^2 P_{\leq
-11}(v)dxdt\right|\\
&&+\sum_{k\geq 1} \left|\int_{\R\times
[0,t_k]}\wt{P}_k(U)\partial_x(u_1+u_2)\partial_x
vdxdt\right|\\
&&+\sum_{k\geq 1} \left|\int_{\R\times [0,t_k]}\wt{P}_k(U)v\cdot
\partial_x^2 u_1dxdt\right|\\
&:=&I+II+III+IV.
\end{eqnarray*}
For the contribution of $I$ we can bound it as in
\eqref{eq:energydiffprf2} and then get that
\[I\les \norm{U}_{F^0(1)}^2\norm{u_2}_{F^{1}(1)}. \]
For the contribution of $II$, since the derivatives fall on the low
frequency, then we can easily get
\[II\les \norm{U}_{F^0(1)}^2\norm{u_2}_{F^{1}(1)}. \]
We consider now the contribution of $IV$.
\begin{eqnarray*}
IV&\les& \sum_{k\geq 1} \sum_{k_1,k_2\in \Z_+}\left|\int_{\R\times
[0,t_k]}\wt{P}_k(U)\cdot \wt{P}_{k_1}(v)\cdot
\partial_x^2 \wt{P}_{k_2}(u_1)dxdt\right|\\
&\les& \sum_{k\geq 1}\sum_{|k-k_2|\leq 5, k_1 \leq k-10}
2^{2k-k_1/2}
\norm{\wt{P}_{k}(U)}_{F_k(1)}\norm{\wt{P}_{k_1}(v)}_{F_{k_1}(1)}\norm{\wt{P}_{k_2}(u_1)}_{F_{k_2}(1)}\\
&&+\sum_{k\geq 1}\sum_{k_1 \geq k-10} 2^{2k_2}2^{-\min(k,k_2)/2}
\norm{\wt{P}_{k}(U)}_{F_k(1)}\norm{\wt{P}_{k_1}(v)}_{F_{k_1}(1)}\norm{\wt{P}_{k_2}(u_1)}_{F_{k_2}(1)}\\
&\les&
\norm{U}_{F^0(1)}\norm{v}_{F^0(1)}\norm{u_1}_{F^{2}(1)}+\norm{U}_{F^0(1)}^2\norm{u_1}_{F^{1}(1)}.
\end{eqnarray*}
For the contribution of $III$, we obtain
\begin{eqnarray*}
III&\les&\sum_{k\geq 1}\sum_{k_1\leq k_2-10} \left|\int_{\R\times
[0,t_k]}\wt{P}_k(U)\partial_x\wt{P}_{k_1}(u_1+u_2)\cdot\partial_x
\wt{P}_{k_2}(v)dxdt\right|\\
&&+ \sum_{k\geq 1}\sum_{k_1\geq k_2-9}\left|\int_{\R\times
[0,t_k]}\wt{P}_k(U)\partial_x\wt{P}_{k_1}(u_1+u_2)\cdot\partial_x
\wt{P}_{k_2}(v)dxdt\right|\\
&\les&\norm{U}_{F^0(1)}^2(\norm{u_1}_{F^1(1)}+\norm{u_2}_{F^1(1)}).
\end{eqnarray*}
Therefore, we have proved that
\begin{eqnarray*}
\norm{U}_{E^0(1)}^2&\les
&\norm{\phi}_{H^{1,0}}^2+\norm{U}_{F^0(1)}^2(\norm{u_1}_{F^1(1)}+\norm{u_2}_{F^1(1)})\\
&&+\norm{U}_{F^0(1)}\norm{v}_{F^0(1)}\norm{u_1}_{F^{2}(1)}.
\end{eqnarray*}
By \eqref{eq:energydiffsmall}, Theorem \ref{thmmain} (a),
\eqref{eq:L2conti} and \eqref{eq:energyv} we get
\[\norm{U}_{E^0(1)}\les \norm{\phi_1-\phi_2}_{H^{1,0}}+\norm{\phi_1-\phi_2}_{L^2}\norm{\phi_1}_{H^{2,0}},\]
which combined with \eqref{eq:energyv} completes the proof of the
proposition.
\end{proof}

\noindent{\bf Acknowledgment.} This work is supported in part by
RFDP of China, Number 20060001010, the National Science Foundation
of China, grant 10571004; and the 973 Project Foundation of China,
grant 2006CB805902, and the Innovation Group Foundation of NSFC,
grant 10621061.

\end{document}